\documentclass[12pt]{article}
\usepackage{epsfig}
\usepackage{amsmath,amsthm,amsfonts,amssymb,amscd,cite}
\usepackage{graphicx}
\usepackage[active]{srcltx}
\usepackage[colorlinks=true, linkcolor=red, filecolor=brown, citecolor=blue, urlcolor=blue]{hyperref}
\usepackage{float}
\newtheorem{proposition}{Proposition}

\allowdisplaybreaks

\begin{document}

\baselineskip=0.30in

\vspace*{30mm}

\begin{center}
{\Large \bf Splice Graphs and Their Topological Indices }

\vspace{10mm}

{\large \bf Reza Sharafdini}$^1$  \ , \, {\large \bf Ivan Gutman}$^2$

\vspace{9mm}

\baselineskip=0.20in

$^1$ {\it Department of Mathematics, Persian Gulf University, \\
Bushehr 7516913817, Iran\/}, \\
e-mail: {\tt sharafdini@pgu.ac.ir} \\[3mm]

$^2$ {\it Faculty of Science, University of Kragujevac, P. O. Box 60, \\ 34000 Kragujevac, Serbia\/}, \\
e-mail: {\tt gutman@kg.ac.rs}

\vspace{6mm}

\end{center}

\vspace{6mm}

\baselineskip=0.20in

\noindent {\bf Abstract }

\vspace{3mm}

{\small Let $G_1=(V_1,E_1)$ and $G_2=(V_2,E_2)$ be two graphs
with disjoint vertex sets $V_1$ and $V_2$. Let $u_1 \in V_1$
and $u_2 \in V_2$.  A splice of $G_1$ and $G_2$ by vertices $u_1$ and $u_2$,  $\mathcal{S}(G_1,G_2;u_1,u_2)$, is defined by identifying the vertices $u_1$ and $u_2$ in the union of $G_1$ and $G_2$. In this paper we calculate the
Szeged, edge-Szeged, $PI$, vertex-$PI$ and eccentric connectivity
indices of splice graphs.}

\vspace{10mm}

\baselineskip=0.30in

\noindent {\large \bf 1. Introduction}

\vspace{3mm}

Let $G=(V,E)$ be a simple graph with $V$ and $E$ being its vertex
and edge sets, respectively. Graph theory has successfully provided
chemists with a variety of useful tools \cite{1,2}, among which are
the topological indices or molecular--graph--based structure
descriptors \cite{3,4}. In this paper we are concerned with five of
these topological indices, that recently attracted much attention and
found noteworthy chemical applications.

All graphs considered in this paper are assumed to be simple, undirected,
without weighted edges, and connected.

The {\bf Szeged index} of a graph $G$ is denoted by $Sz(G)$ and defined
as \cite{5}
\begin{equation}                    \label{1}
Sz = Sz(G) = \sum_{e=uv} n_u(e)\,n_v(e) \ .
\end{equation}
Here the sum is taken over all edges of $G$, and for a given edge $e=uv$,
the quantity $n_u(e)$ denotes the number of vertices closer to $u$ than to
$v$, and the quantity $n_v(e)$ is defined analogously. For more details on
the Szeged index see the review \cite{6} and the references cited therein.

Denote by $d_G(x,y)$ the distance (= number of edges in a shortest path)
between the vertices $x$ and $y$ of the graph $G$. Then we can define
the sets
$$
N(e,u,G) = \Big\{x \in V(G)\,|\,d_G(x,u)<d_G(x,v)\Big\}
$$
and
$$
N(e,v,G) = \Big\{x \in V(G)\,|\,d_G(x,v)<d_G(x,u)\Big\}
$$
by means of which we have:
$$
n_u(e) = n_u(e,G) = \Big| N(e,u,G)\Big| \hspace{10mm} \mbox{and} \hspace{10mm}
n_v(e) = n_v(e,G) = \Big| N(e,v,G)\Big| \ .
$$

It is obvious that an end-vertex of any edge is closer to itself
than to the other end-vertex of that edge. Therefore the product
$n_u(e)\,n_v(e)$ is always positive.

The {\bf edge-Szeged index} is obtained by replacing
$n_u(e)\,n_v(e)$ in Eq. (\ref{1}) by $m_u(e)\,m_v(e)$, where
$m_u(e)$ is the number of edges in $G$ whose distance to vertex
$u$ is smaller than the distance to vertex $v$, and $m_v(e)$ is
defined analogously. Hence the edge version of the Szeged index is
given by \cite{7}
$$
Sz_e = Sz_e(G) = \sum_{e=uv} m_u(e)\,m_v(e) \ .
$$
We recall that the distance between the edge $f=xy$ and the vertex
$u$ in the graph $G$, denoted by $d_G(f,u)$, is define as
$d_G(f,u) = \min \Big\{ d_G(x,u),d_G(y,u)\Big\}$ . We can now
introduce the sets
$$
M(e,u,G) = \Big\{f \in E(G)\,|\,d_G(f,u)<d_G(f,v)\Big\}
$$
and
$$
M(e,v,G) = \Big\{f \in E(G)\,|\,d_G(f,v)<d_G(f,u)\Big\}
$$
by means of which we have:
$$
m_u(e) = m_u(e,G) = \Big| M(e,u,G)\Big| \hspace{10mm} \mbox{and} \hspace{10mm}
m_v(e) = m_v(e,G) = \Big| M(e,v,G)\Big| \ .
$$

If instead of multiplicative contributions $n_u(e)\,n_v(e)$
and $m_u(e)\,m_v(e)$, we consider their additive versions,
$n_u(e)+n_v(e)$ and $m_u(e)+m_v(e)$, then we obtain the
vertex-- and the edge-$PI$ indices, respectively.\footnote{The
acronym $PI$ comes from Padmakar and Ivan. Whereas Padmakar
is the first name of P. V. Khadikar, the inventor of the
$PI$ index \cite{8}, Ivan comes from the first name of one
of the present authors, whose contribution to the discovery
of the $PI$ index was nil.}

The {\bf edge-$\mathbf{PI}$ index} is defined as \cite{8}
$$
PI_e = PI_e(G) = \sum_{e=uv} \Big[ m_u(e) + m_v(e) \Big] \ .
$$
Since this edge version was introduced first, the subscript $e$ is
usually omitted and the index is referred to simply as $PI$ index. More
details on the $PI$ index are found in the review \cite{9} and the
references cited therein.

The {\bf vertex-$\mathbf{PI}$ index} seems to was first considered
by Khalifeh et al. \cite{10} and is defined as
$$
PI_v = PI_v(G) = \sum_{e=uv} \Big[ n_u(e) + n_v(e) \Big] \ .
$$

The {\bf eccentric connectivity index} of the graph $G$ is defined
as \cite{11}
\begin{equation}                    \label{ec}
Ecc = Ecc(G) = \sum_{u \in V} \deg(u)\,\varepsilon(u)
\end{equation}
where for a given vertex $u$, its eccentricity $\varepsilon(u)$
is the greatest distance between $u$ and any other vertex of $G$.
The maximum eccentricity over all vertices of $G$ is called the
{\it diameter\/} of $G$ whereas the minimum eccentricity among the
vertices of $G$ is the {\it radius\/} of $G$. The set of vertices
whose eccentricity is equal to the radius of $G$ is called the
{\it center\/} of $G$. It is well known that each tree has either
one or two vertices in its center.

The aim of this article is to contribute to the theory of the
above described five topological indices by showing how these can
be computed in the case of splice graphs.

Suppose that $G_1=(V_1,E_1)$ and $G_2=(V_2,E_2)$ are two graphs
with disjoint vertex sets. Let $u_1 \in V_1$ and $u_2 \in V_2$ be
given vertices of $G_1$ and $G_2$\,, respectively.
Following $\rm Do\check{s}li\acute{c}$ \cite{Doslic-splice}, a splice of $G_1$ and $G_2$ by vertices $u_1$ and $u_2$,  $\mathcal{S}(G_1,G_2;u_1,u_2)$, is defined by identifying the vertices $u_1$ and $u_2$ in the union of $G_1$ and $G_2$ (See Figure 1).

\begin{figure}[H]
\centering
\includegraphics[height=3cm]{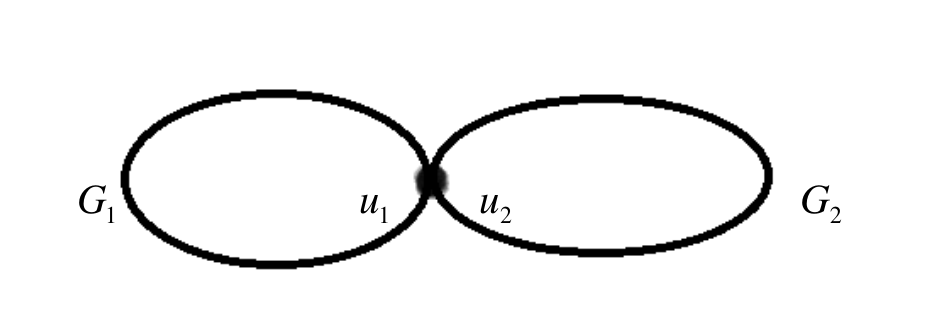}
\caption{The splice graph $\mathcal{S}(G_1,G_2;u_1,u_2)$ and
the labeling of its structural details.}
\end{figure}

\baselineskip=0.20in

Some topological indices of splice graphs are computed already \cite{Ashrafi-splice,balak-cutvetex,Mansour-Br-sp,Reza-Hosoya}. In this paper we aim to compute some other indices for these graphs. Note that the idea of the proofs is based on that of a previously communicated work \cite{bridge-Mogh-Gutman}.

\baselineskip=0.30in

\vspace{5mm}

\baselineskip=0.30in

\noindent {\large \bf 2. Main Results}

\vspace{3mm}

In this section, we compute the Szeged, edge-Szeged, edge-$PI$,
vertex-$PI$, and eccentric connectivity indices of the above
described splice graph. We first introduce the following
structural parameters of $\mathcal{S} = \mathcal{S}(G_1,G_2;u_1,u_2)$. Let $i=1,2$ and
$f=xy \in E_i$\,, and
\begin{eqnarray*}
n_x^i(f) = \Big|N(f,x,G_i) \Big| & \ , \ & n_y^i(f) =
\Big|N(f,y,G_i) \Big| \\
n_x(f) = \Big|N(f,x,\mathcal{S}) \Big| & \ , \ & n_y(f) =
\Big|N(f,y,\mathcal{S}) \Big| \\
m_x^i(f) = \Big|M(f,x,G_i) \Big| & \ , \ & m_y^i(f) =
\Big|M(f,y,G_i) \Big| \\
m_x(f) = \Big|M(f,x,\mathcal{S}) \Big| & \ , \ & m_y(f) = \Big|M(f,y,\mathcal{S})
\Big| \ .
\end{eqnarray*}
In addition, for a given vertex $u \in V(\mathcal{S})$, let
$\varepsilon_1(u)$ be the eccentricity of $u$ as a vertex of
$G_1$\,, $\varepsilon_2(u)$ the eccentricity of $u$ as a vertex of
$G_2$\,, and $\varepsilon(u)$ the eccentricity of $u$ as a vertex
of the splice graph $\mathcal{S}$\,.

\vspace{3mm}

\begin{proposition} \label{prop1} Assume that $f=xy \in E_1$\,. \\
({\it i}) If $u_1 \in N(f,x,G_1)$, then
\begin{eqnarray*}
n_x(f)=n_x^1(f) + |V_2| & \ , \ & n_y(f)=n_y^1(f) \\
m_x(f)=m_x^1(f) + |E_2|& \ , \ & m_y(f)=m_y^1(f) \ .
\end{eqnarray*}
({\it ii}) If $d_{G_1}(u_1,x) = d_{G_1}(u_1,y)$, then
\begin{eqnarray*}
n_x(f)=n_x^1(f) & \ , \ & n_y(f)=n_y^1(f) \\
m_x(f)=m_x^1(f) & \ , \ & m_y(f)=m_y^1(f) \ .
\end{eqnarray*}
Analogous relations hold if $f=xy \in E_2$\,.
\end{proposition}

\vspace{3mm}

\noindent {\it Proof\/} is easy and is left to the reader.

\vspace{3mm}

In order to verify the following propositions we need some
preparations. For $i=1,2$, let
$$
S_i = \Big\{ f=xy \in E_i\,|\,d_{G_i}(x,u_i) = d_{G_i}(y,u_i)
\Big\} \hspace{3mm} , \hspace{3mm} T_i = E_i \setminus S_i
\hspace{3mm} , \hspace{3mm} t_i = |T_i| \ .
$$
In addition, for $f=xy \in E_i$ define
$$
n^i(f) = \left\{
\begin{array}{lll}
n_x^i(f) & \ \mbox{ if } \ & d_{G_i}(x,u_i) > d_{G_i}(y,u_i) \\[3mm]
n_y^i(f) & \ \mbox{ if } \ & d_{G_i}(x,u_i) < d_{G_i}(y,u_i) \\[3mm]
0 & \ \mbox{ if } \ & d_{G_i}(x,u_i) = d_{G_i}(y,u_i)
\end{array}
\right.
$$
and
$$
m^i(f) = \left\{
\begin{array}{lll}
m_x^i(f) & \ \mbox{ if } \ & d_{G_i}(x,u_i) > d_{G_i}(y,u_i) \\[3mm]
m_y^i(f) & \ \mbox{ if } \ & d_{G_i}(x,u_i) < d_{G_i}(y,u_i) \\[3mm]
0 & \ \mbox{ if } \ & d_{G_i}(x,u_i) = d_{G_i}(y,u_i) \ .
\end{array}
\right.
$$

\vspace{3mm}

\begin{proposition} \label{prop2}
\begin{eqnarray*}
Sz(\mathcal{S}) & = & Sz(G_1)+Sz(G_2) \\[3mm]
& + & |V_2|\,\sum_{f=xy \in E_1} n^1(f) + |V_1|\,\sum_{f=xy \in
E_2} n^2(f)\ .
\end{eqnarray*}
\end{proposition}

\begin{proof}
\begin{eqnarray}
Sz(\mathcal{S}) & = & \sum_{f=xy \in E_1} n_x(f)\,n_y(f) + \sum_{f=xy \in E_2} n_x(f)\,n_y(f)\nonumber \\[3mm]
& = & \sum_{f=xy \in E_1} n_x^1(f)\,n_y^1(f) + \sum_{f=xy \in E_1} |V_2|\,n^1(f) \nonumber \\[3mm]
& + & \sum_{f=xy \in E_2} n_x^2(f)\,n_y^2(f) + \sum_{f=xy \in E_2} |V_1|\,n^2(f) \ .     \label{p2}
\end{eqnarray}
Because for $i=1,2$,
$$
\sum_{f=xy \in E_i} n_x^i(f)\,n_y^i(f) = Sz(G_i)
$$
from Eq. (\ref{p2}) we directly obtain Proposition \ref{prop2}. \end{proof}

\vspace{3mm}

\begin{proposition} \label{prop3}
\begin{eqnarray*}
Sz_e(\mathcal{S}) & = & Sz_e(G_1)+Sz_e(G_2) \\[3mm]
& + & |E_2|\,\sum_{f=xy \in E_1} m^1(f) +
|E_1|\,\sum_{f=xy \in E_2} m^2(f) \ .
\end{eqnarray*}
\end{proposition}

\begin{proof}
\begin{eqnarray}
Sz_e(\mathcal{S}) & = & \sum_{f=xy \in E_1} m_x(f)\,m_y(f) + \sum_{f=xy \in E_2} m_x(f)\,m_y(f)\nonumber \\[3mm]
& = & \sum_{f=xy \in E_1} m_x^1(f)\,m_y^1(f) + \sum_{f=xy \in E_1} |E_2|\,m^1(f) \nonumber \\[3mm]
& + & \sum_{f=xy \in E_2} m_x^2(f)\,m_y^2(f) + \sum_{f=xy \in E_2} |E_1|\,m^2(f)\ .     \label{p3}
\end{eqnarray}
Because for $i=1,2$,
$$
\sum_{f=xy \in E_i} m_x^i(f)\,m_y^i(f) = Sz_e(G_i)
$$
from Eq. (\ref{p3}) we directly obtain Proposition \ref{prop3}.
\end{proof}

\begin{proposition} \label{prop4}
$$
PI_v(\mathcal{S}) = PI_v(G_1)+PI_v(G_2) + (t_2+1)|V_1| + (t_1+1)|V_2| \ .
$$
\end{proposition}

\begin{proof}
\begin{eqnarray}
PI_v(\mathcal{S}) & = & \sum_{f=xy \in T_1} [n_x(f) + n_y(f)] + \sum_{f=xy \in S_1} [n_x(f) + n_y(f)]
\nonumber \\[3mm]
& + & \sum_{f=xy \in T_2} [n_x(f) + n_y(f)] + \sum_{f=xy \in S_2} [n_x(f) + n_y(f)]
+ [n_{u_1}(e) + n_{u_2}(e)] \nonumber \\[3mm]
& = & \sum_{f=xy \in T_1} [n_x^1(f) + |V_2| + n_y^1(f)] +
\sum_{f=xy \in S_1} [n_x^1(f) + n_y^1(f)] \nonumber \\[3mm]
& + & \sum_{f=xy \in T_2} [n_x^2(f) + |V_1| + n_y^2(f)] +
\sum_{f=xy \in S_2} [n_x^2(f) + n_y^2(f)] + (|V_1|+|V_2|) \nonumber \\[3mm]
& = & \sum_{f=xy \in T_1} [n_x^1(f) + n_y^1(f)] + t_1\,|V_2| +
\sum_{f=xy \in S_1} [n_x^1(f) + n_y^1(f)] \nonumber \\[3mm]
& + & \sum_{f=xy \in T_2} [n_x^2(f) + n_y^2(f)] + t_2\,|V_1| +
\sum_{f=xy \in S_2} [n_x^2(f) + n_y^2(f)] \nonumber \\[3mm]
& + & (|V_1|+|V_2|) \ .      \label{p4}
\end{eqnarray}
Because for $i=1,2$,
$$
\sum_{f=xy \in T_i} [n_x^i(f) + n_y^i(f)] + \sum_{f=xy \in S_i} [n_x^i(f) + n_y^i(f)]
= \sum_{f=xy \in E_i} [n_x^i(f) + n_y^i(f)] = PI_v(G_i)
$$
from Eq. (\ref{p4}) we directly obtain Proposition \ref{prop4}. \end{proof}

\vspace{3mm}

\begin{proposition} \label{prop5}
$$
PI(\mathcal{S}) = PI(G_1)+PI(G_2) + t_2|E_1| + t_1|E_2|
\ .
$$
\end{proposition}

\begin{proof}

\begin{eqnarray}
PI(\mathcal{S}) & = & \sum_{f=xy \in T_1} [m_x(f) + m_y(f)] + \sum_{f=xy \in S_1} [m_x(f) + m_y(f)]
\nonumber \\[3mm]
& + & \sum_{f=xy \in T_2} [m_x(f) + m_y(f)] + \sum_{f=xy \in S_2} [m_x(f) + m_y(f)]
+ [m_{u_1}(e) + m_{u_2}(e)] \nonumber \\[3mm]
& = & \sum_{f=xy \in T_1} [m_x^1(f) + |E_2|+ m_y^1(f)] +
\sum_{f=xy \in S_1} [m_x^1(f) + m_y^1(f)] \nonumber \\[3mm]
& + & \sum_{f=xy \in T_2} [m_x^2(f) + |E_1|+ m_y^2(f)] +
\sum_{f=xy \in S_2} [m_x^2(f) + m_y^2(f)] \nonumber \\[3mm]
& = & \sum_{f=xy \in T_1} [m_x^1(f) + m_y^1(f)] + t_1|E_2| +
\sum_{f=xy \in S_1} [m_x^1(f) + m_y^1(f)] \nonumber \\[3mm]
& + & \sum_{f=xy \in T_2} [m_x^2(f) + m_y^2(f)] + t_2|E_1| +
\sum_{f=xy \in S_2} [m_x^2(f) + m_y^2(f)] \nonumber
\ .      \label{p5}
\end{eqnarray}
Because for $i=1,2$,
$$
\sum_{f=xy \in T_i} [m_x^i(f) + m_y^i(f)] + \sum_{f=xy \in S_i} [m_x^i(f) + m_y^i(f)]
= \sum_{f=xy \in E_i} [m_x^i(f) + m_y^i(f)] = PI(G_i)
$$
from Eq. (\ref{p5}) we directly obtain Proposition \ref{prop5}. \end{proof}

\begin{proposition} \label{prop-Ecc}
\begin{eqnarray*}
Ecc(\mathcal{S}) & = & \sum_{x \in V_1}  \deg_\mathcal{S} (x) \cdot \max \Big\{
d_{G_1}(x,u_1)+\varepsilon_2(u_2),\,\varepsilon_1(x) \Big\}
\\[3mm]
& + & \sum_{y \in V_2}  \deg_\mathcal{S} (y) \cdot \max \Big\{
d_{G_2}(y,u_2)+\varepsilon_1(u_1)\,,\,\varepsilon_2(y) \Big\} \
.
\end{eqnarray*}
\end{proposition}

\begin{proof}
Consider a vertex $x$ of the splice graph $\mathcal{S}$,
such that $x \in V_1$. Let $z$ be the vertex of $\mathcal{S}$ whose distance to $x$
is maximal. Thus, $\varepsilon(x)=d_{\mathcal{S}}(x,z)$. If $z$ belongs to the
graph $G_1$\,, then $\varepsilon(x)=\varepsilon_1(x)$. If $z$ belongs to
the graph $G_2$\,, then the distance between $x$ and $z$ is equal to
$d_{G_1}(x,u_1)+d_{G_2}(u_2,z)$, see Fig. 1. In addition, $z$ must be the
vertex at greatest distance from $u_2$\,. Consequently,
$\varepsilon(x)=d_{G_1}(x,u_1)+\varepsilon_2(u_2)$.
This means that no matter where the vertex $z$ is located,
$$
\varepsilon(x)=\max \Big\{ d_{G_1}(x,u_1)+\varepsilon_2(u_2)\,,\,\varepsilon_1(x) \Big\}
$$
holds for any vertex $x \in V_1$. The formula for $\varepsilon(y)$ in the case
when $y \in V_2$ is fully analogous.

Proposition \ref{prop-Ecc} follows now from the definition of
the eccentric connectivity index, Eq. (\ref{ec}). \end{proof}

\end{document}